\documentclass{amsart}

\usepackage{amsmath, amsthm,amssymb, amsfonts}
\usepackage{rotating}
\usepackage{url}
\usepackage{microtype}
\usepackage{cite}

\def\A{{\mathcal A}}

\renewcommand{\phi}{\varphi}

\def\Z{{\mathbb Z}}

\def\Q{{\mathbb Q}}
\def\F{{\mathbb F}}

\def\pair{{\langle\ ,\ \rangle}}
\def\sqpair{{[ \ , \ ]}}

\def\Hom{{\operatorname{Hom}}}
\def\rank{{\operatorname{rank}}}
\def\Sym{{\operatorname{Sym}}}
\def\Aut{{\operatorname{Aut \;}}}
\def\GL{{\operatorname{GL}}}
\def\Sur{{\operatorname{Sur}}}
\def\Jac{{\operatorname{Jac}}}
\def\coker{{\operatorname{coker}}}

\def\ra{{\rightarrow}}

\newtheorem{thm}{Theorem}
\newtheorem{conj}{Conjecture}
\newtheorem{prop}[thm]{Proposition}

\newtheorem{lemma}[thm]{Lemma}

\newtheorem{heuristic}[thm]{Heuristic}

\theoremstyle{remark}

\newtheorem{rem}{Remark}

\title{On a Cohen-Lenstra Heuristic for Jacobians of Random Graphs}
\author[Julien Clancy, Nathan Kaplan, Timothy Leake, Sam Payne, Melanie Wood]{Julien Clancy, Nathan Kaplan, Timothy Leake, Sam Payne, Melanie Matchett Wood}

\begin{document}

\begin{abstract}
In this paper, we make specific conjectures about the distribution of Jacobians of random graphs with their canonical duality pairings.  Our conjectures are based on a Cohen-Lenstra type heuristic saying that a finite abelian group with duality pairing appears with frequency inversely proportional to the size of the group times the size of the group of automorphisms that preserve the pairing.  
We conjecture that the Jacobian of a random graph is cyclic with probability a little over $.7935$.
We determine the values of several other statistics on Jacobians of random graphs that would follow from our conjectures.    In support of the conjectures, we prove that random symmetric matrices over $\Z_p$, distributed according to Haar measure, have cokernels distributed according to the above heuristic.
We also give experimental evidence in support of our conjectures. 
\end{abstract}

\maketitle

\section{Introduction}

Jacobians of graphs are often cyclic.  A similar phenomenon has been observed in class groups of imaginary quadratic fields, where it is conjecturally explained by the classical Cohen-Lenstra heuristic, in which a group $\Gamma$ appears with frequency proportional to $1/ \# \Aut \Gamma $.  Jacobians of 
Erd\H{o}s--R\'{e}nyi 
random graphs also seem to exhibit some of the deeper properties predicted by this heuristic. For instance, we have observed empirically that the average size of the Jacobian of a random graph modulo $p$ tends to 2, for all primes $p$, matching \cite{CohenLenstra84}.  However, we have also observed that the odd part of the Jacobian of a random graph is cyclic with probability close to $.946$, which does not match the classical Cohen-Lenstra heuristic prediction that the odd part of a random abelian group should be cyclic with probability a little over $.9775$.  This paper shows how these and other observed phenomena are explained by a natural variation on the Cohen-Lenstra heuristic, proposed in \cite{CLP13}, based on the fact that the Jacobian of a graph carries a canonical duality pairing.

\begin{heuristic}[\!\!\cite{CLP13}]\label{H:Main}
A group $\Gamma$ with pairing $\delta$ occurs as a Jacobian of a random graph with frequency proportional to $$
\frac{1}{ \# \Gamma \cdot \# \Aut(\Gamma,\delta)} \; ,
$$
where $\Aut(\Gamma,\delta)$ denotes the group of automorphisms of $\Gamma$ that respect the pairing $\delta$.
\end{heuristic}

\noindent In the present paper, we use this heuristic to make precise conjectures and compute predicted averages based on these conjectures for many specific statistics. 
The Jacobian of a graph is the torsion part of the cokernel of its Laplacian matrix, which is a symmetric matrix, and we also prove
results about random symmetric matrices distributed according to Haar measure in $\Z_p$, showing their cokernels are distributed according to Heuristic~\ref{H:Main}.  We also present empirical data to support the conjectures, in Section~\ref{S:Data}.

\subsection{The pairing}
Recall that a duality pairing on a finite abelian group $\Gamma$ is a symmetric bilinear map $\delta:\ \Gamma \times \Gamma \rightarrow \Q/\Z$ such that the induced map $g\rightarrow \langle g,\ \rangle$ is an isomorphism from $\Gamma$ to $\Hom(\Gamma,\Q/\Z)$.  The cokernel of a nonsingular symmetric integer matrix $A$ carries a canonical duality pairing, induced by
\[
\langle x , y \rangle = y^t A^{-1} x.
\]
More generally, the torsion part of the cokernel of any symmetric integer matrix carries a canonical duality pairing.  The Jacobian of a graph occurs naturally in this way, as the torsion subgroup of the cokernel of the combinatorial Laplacian. See \cite{Shokrieh10} for a detailed discussion of the duality pairing on graph Jacobians and its relation to the Grothendieck pairing, or monodromy pairing, on component groups of N\'eron models.  

\subsection{Conjectures}

Let $G(n,q)$ be the Erd\H{o}s--R\'{e}nyi random graph on $n$ vertices, where each edge is included independently with probability $q$, for some fixed probability $0 < q < 1$.  In other words, $G(n,q)$ is the probability space on graphs with $n$ vertices in which a graph $G$ with $e$ edges appears with probability $q^{e} \cdot (1-q)^{{n \choose 2} - e}$.  Here we study the associated probability space on isomorphism classes of finite abelian groups with duality pairing, 
\[
\Gamma(n,q) = \Jac(G(n,q)),
\]
in which the measure of a subset is the probability that the Jacobian of a random graph in $G(n,q)$ lies in that subset.

Let $\A(m)$ be the set of all isomorphism classes of pairs $(\Gamma,\delta)$, where $\Gamma$ is an abelian group of order $m$ and $\delta$ is a duality pairing on $\Gamma$.  Our first conjecture is the analog of Cohen and Lenstra's Fundamental Assumption 8.1 for their heuristics on class groups of number fields \cite{CohenLenstra84} .

\begin{conj}\label{C:basic}
 Let $F$ be a function on isomorphism classes of finite abelian groups with duality pairings that is either bounded or depends only on the Sylow $p$-subgroups of $F$ for finite many $p$.  
Then
\[
\lim_{n \rightarrow \infty} \mathbb{E}(F(\Gamma(n,q))) = \lim_{n\rightarrow \infty}
\frac{\sum_{m=1}^n \sum_{(\Gamma,\delta)\in \A(m)}  \frac{F(\Gamma,\delta)}{ \# \Gamma \cdot \# \Aut(\Gamma,\delta)}}{\sum_{m=1}^n \sum_{(\Gamma,\delta)\in \A(m)}  \frac{1}{\# \Gamma \cdot \# \Aut(\Gamma,\delta)}}
\; .
\]
\end{conj}

Since we include all bounded test functions $F$, Conjecture~\ref{C:basic} thus includes a claim in the spirit of weak convergence. However, note that neither side can be expressed as the evaluation of $F$ against some measure $\nu$ on the set of finite abelian groups, since when $F$ is the characteristic function of a group, the left-hand side of Conjecture~\ref{C:basic} is $0$ \cite[Corollary 9.3]{Wood14}, and from the product of Proposition~\ref{P:normalizing} over all $p$, if follows the right-hand side is $0$, which would contradict countable additivity of $\nu$.  Even though Cohen and Lenstra said any non-negative test function $F$ should ``probably'' be included in their analogous conjecture on class groups \cite[8.1]{CohenLenstra84}, it is likely that is too much to hope for in the case of class groups, and it is definitely too much to hope for in the case of Jacobians of random graphs.  For example, \cite[Theorem 5]{GJRWW} shows that $(\Z/2\Z)^k$ is never a Jacobian of a graph.  So if we take a function $F$ supported on these groups and growing fast enough that the limit on the right-hand side is positive, 
Conjecture~\ref{C:basic} would fail for that $F$.

Any finite abelian group with pairing splits as an orthogonal direct sum of its Sylow $p$-subgroups, and many interesting functions, such as the indicator function of the set of cyclic groups with pairing, depend only on their values on the Sylow $p$-subgroups.  One important special case is where the function $F$ depends only on the Sylow $p$-subgroup of $\Gamma$ with its restricted pairing, for a single fixed prime $p$.  In this case, Conjecture
~\ref{C:basic} implies the following, as in \cite[Proposition 5.6]{CohenLenstra84}:
\begin{equation}\label{E:pSylow}
\lim_{n \rightarrow \infty} \mathbb{E}(F(\Gamma(n,q))) = 
\frac{\sum_{m=0}^\infty \sum_{(\Gamma,\delta)\in \A(p^m)}  \frac{F(\Gamma,\delta)}{\# \Gamma\cdot \# \Aut(\Gamma,\delta)}}{\sum_{m=0}^\infty \sum_{(\Gamma,\delta)\in \A(p^m)}  \frac{1}{\# \Gamma\cdot \# \Aut(\Gamma,\delta)}}
.
\end{equation}
In Proposition~\ref{P:normalizing}, we show that the denominator of the right-hand side above converges to $\prod_{i=1}^\infty (1-p^{1-2i})^{-1}$.  Therefore, we can put a measure $\mu$ on
\[
\A_p =\bigcup_m \A(p^m)
\]
so that $$\mu(\Gamma,\delta) =\frac{\prod_{i=1}^\infty (1-p^{1-2i})}{\# \Gamma\cdot \# \Aut(\Gamma,\delta)}.$$  
(Note that as explained above, there is no way to have an analogous measure of the set of all finite abelian groups.)
 Then, if $F$ depends only on the Sylow $p$-subgroup with its restricted pairing, Conjecture~\ref{C:basic} says that
\begin{equation}\label{E:int}
\lim_{n \rightarrow \infty} \mathbb{E}(F(\Gamma(n,q))) = \int_{(\Gamma,\delta)\in \A_p} F(\Gamma,\delta) d\mu.
\end{equation}
As with the classical Cohen-Lenstra heuristic, different functions $F$ give rise to estimates for various statistics on random finite abelian $p$-groups with duality pairings, distributed according to Heuristic~\ref{H:Main}.  

In Section~\ref{S:averages}, we compute the integral on the right-hand side above for several interesting functions $F$, the indicator function for the set of groups with trivial $p$-part, the indicator function for groups with cyclic $p$-part, the number of surjections onto a fixed group, and the functions $p^{k r_p(\Gamma)}$, where $r_p(\Gamma)$ is the $p$-rank of $\Gamma$, i.e. the rank of the free $\Z/p\Z$-module $\Gamma \otimes \Z/p\Z$.  For example, in Theorem~\ref{T:momint}, we show that
if $\Gamma'=\prod_{i=1}^r \Z/p^{e_i}\Z$ with $e_1\le e_2 \le \cdots \le e_r$ then 
$$\int_{(\Gamma,\delta)\in\A_p} \#\Sur(\Gamma,\Gamma') d\mu =p^{(r-1)e_1 + (r-2)e_2+\cdots + e_{r-1}}.$$

\begin{rem}
While this paper was in preparation, the fifth author proved Conjecture~\ref{C:basic} for many functions $F$ that depend on only finitely many Sylow $p$-subgroups of $\Gamma$ \cite{Wood14}.  In particular, our predictions are now confirmed for $F$ the indicator function for the set of groups with a given Sylow $p$-subgroup, the indicator function for groups with cyclic $p$-part, the number of surjections onto a fixed group, and   $p^{k r_p(\Gamma)}$.  Theorem~\ref{thm:pparts} from this paper is an ingredient in the proof.
\end{rem}

In Proposition~\ref{P:cyclic}, we show that Conjecture~\ref{C:basic} implies that the asymptotic probability that that Sylow $p$-subgroup of the Jacobian of a random graph is cyclic is $\prod_{i=1}^{\infty} (1-p^{-1-2i})$.  Taking the product over all primes $p$, we are led to the following conjecture.  

\begin{conj}\label{C:cyc}
The probability that the Jacobian of $G(n,q)$ is cyclic tends to
\[
\prod_{p} {\prod_{i=1}^\infty (1 - p^{-1-2i})}=\zeta(3)^{-1}\zeta(5)^{-1}\zeta(7)^{-1}\zeta(9)^{-1}\cdots
\] 
as $n$ goes to infinity, where $\zeta(s)$ is the Riemann zeta function.
\end{conj}

\noindent This product converges to approximately $.7935$.  Note that Wagner has made other conjectures on statistics of Jacobians of random graphs \cite[Conjectures 4.2, 4.3, 4.4]{Wagner00}, and his conjecture for the asymptotic probability that the Jacobian is cyclic differs from Conjecture~\ref{C:cyc}.

\begin{rem}
One could also consider stronger versions of these conjectures, allowing the probability $q$ to vary with $n$.  If $q$ is too close to 0, then the graph will have very few edges.  In particular, if it has no cycles, then the Jacobian will be trivial.  Similarly, if $q$ is too close to $1$, then $G(n,q)$ will be very close to the complete graph, whose Jacobian is $(\Z/n\Z)^{n-2}$.  Since the $p$-rank of a group changes by at most 1 when an edge is added or deleted \cite[Lemma~5.3]{Lorenzini89}, the Jacobian of $G(n,q)$ will rarely be cyclic when $q$ is too close to 1.  It is natural to expect that versions of these conjectures will hold for $G(n,q(n))$ provided that $q(n) \log(n)$ and $(1- q(n)) \log (n)$ both go to infinity.
\end{rem}

\begin{rem}
The idea of a variation of the classical Cohen-Lenstra heuristic that involves $\# \Gamma$ and the number of automorphisms preserving a bilinear form has appeared earlier, in other contexts.  Cohen and Lenstra already considered a heuristic for class groups of real quadratic fields in which $\Gamma$ appears with frequency proportional to $1/(\# \Gamma\cdot \# \Aut \Gamma)$ \cite{CohenLenstra84, CohenLenstra84b}, and Delaunay studied variations on this heuristic for groups with alternating pairings \cite{Delaunay01, Delaunay07}.  He proposed a heuristic for Tate-Shafarevich groups, which are only conjecturally finite but, when finite, carry a canonical non-degenerate alternating bilinear form $\beta$ \cite{Cassels62}.  In Delaunay's heuristic for Tate-Shafarevich groups of rank $0$ elliptic curves, a finite abelian group with non-degenerate alternating bilinear form $(\Gamma,\beta)$ appears with frequency proportional to $\# \Gamma/ \# \Aut(\Gamma,\beta)$ where $\Aut(\Gamma,\beta)$ is the subgroup of $\Aut(\Gamma)$ consisting of automorphisms that respect the alternating form.  One slight increase in subtlety in our case is that a finite abelian group can carry several non-isomorphic duality pairings \cite{BannaiMunemasa98}, while a group with a non-degenerate alternating bilinear form carries exactly one isomorphism class of such forms.
\end{rem}

\begin{rem}
We do not know any obvious explanation for the exact form of Heuristic~\ref{H:Main}, but some natural analogies are at least suggestive of a relation to the heuristic of Cohen and Lenstra for class groups of real quadratic fields.  Delaunay's heuristic for Tate-Shafarevich groups is motivated by an analogy relating the Mordell-Weil group of an elliptic curve $E$ over $\Q$ to the group of units in a number field.  The analogy readily extends to graphs, with the Mordell-Weil group of an elliptic curve corresponding to the full cokernel of the combinatorial Laplacian, which is a finitely generated abelian group of rank $1$.  The rank $1$ case for elliptic curves corresponds to the case of real quadratic fields, where Cohen and Lenstra predict that a group $\Gamma$ appears with frequency proportional to $1/ ( \# \Gamma \cdot \# \Aut \Gamma)$.  The difference for Jacobians of graphs is that we consider only automorphisms that respect the duality pairing.
\end{rem}

\subsection{Result for Haar random symmetric matrices}

In support of the above conjectures, we prove the following, where a random $p$-adic symmetric matrix stands in for the graph Laplacian.  For a ring $R$, let $\Sym_n(R)$ denote the additive group of symmetric $n\times n$ matrices with coefficients in $R$.

\begin{thm} \label{thm:pparts}
Let $\Gamma$ be a finite abelian $p$-group of rank $r$ with a duality pairing $\delta$, and let $A$ be a random $n \times n$ symmetric matrix with respect to additive Haar measure on $\Sym_n(\Z_p)$.  Then the probability that the cokernel of $A$ with its given duality pairing is isomorphic to $(\Gamma, \delta)$ is
\[
\mu_n(\Gamma,\delta) = \frac{\prod_{j=n-r+1}^n (1-p^{-j}) \prod_{i=1}^{\lceil(n-r)/2\rceil} (1-p^{1-2i})}{\# \Gamma \cdot \# \Aut(\Gamma,\delta)},
\]
where $\Aut(\Gamma,\delta)$ is the set of automorphisms of $\Gamma$ that preserve the pairing $\delta$.  In particular,
\[
\lim_{n\rightarrow \infty} \mu_n(\Gamma,\delta)  = \frac{\prod_{i=1}^\infty (1-p^{1-2i})}{\# \Gamma \cdot \# \Aut(\Gamma,\delta)}.
\]
\end{thm}

\medskip

\noindent In other words, the probability that the cokernel of a random symmetric $n \times n$ matrix over $\Z_p$ is isomorphic to $(\Gamma, \delta)$ tends to a limit that is inversely proportional to $\# \Gamma \cdot \# \Aut(\Gamma, \delta)$, with constant of proportionality $\prod_{i = 1}^\infty (1-p^{1-2i})$.

\begin{rem}
Note that a random matrix is nonsingular with probability 1, but the Laplacian of a random graph is always singular because its row sums and column sums are zero.  However, we also prove that the same distribution holds for the torsion part of the cokernel of a random matrix with all row sums and column sums equal to zero, with respect to Haar measure on this subgroup of $\Sym_n(\Z_p)$.  See Theorem~\ref{thm:zerosum}.
\end{rem}

\begin{rem}
Theorem \ref{thm:pparts} gives the probability that a random matrix in $\Sym_n(\Z_p)$ has cokernel with its given duality pairing isomorphic to a particular pair $(\Gamma,\delta)$.  It is not obvious from the form of this result that the sum taken over all possible pairs of a $p$-group and a duality pairing on it is equal to $1$.  In \cite{Fulman14}, Fulman shows that this defines a probability measure by relating it to the theory of Hall-Littlewood polynomials.  He further shows how this measure occurs as one specialization of a two parameter family of probability measures.  A $p$-group naturally defines a partition $\lambda$, so this result gives a probability measure on the set of all partitions.  Fulman uses two more of these specializations to compute the probability that a partition chosen from this distribution has given size and the probability that a randomly chosen partition has a specified number of parts, which is also proven algebraically in \cite{Wood14}.
\end{rem}

\section{Distribution of Cokernels of Haar Random Symmetric Matrices}\label{sec:cokernels}

In this section, we determine the distribution of cokernels of random symmetric matrices over $\Z_p$, chosen according to Haar measure, and prove Theorem~\ref{thm:pparts}.  In particular, we see that these cokernels are distributed according to Heuristic~\ref{H:Main} as the size of the matrix goes to infinity.  This is an analog of a result of Friedman and Washington \cite{FriedmanWashington89}, that cokernels of Haar random square matrices over $\Z_p$ are distributed according to the Cohen-Lenstra heuristics as their size goes to infinity, as well as an analog of a result of Bhargava, Kane, Lenstra, Poonen, and Rains, that cokernels of skew-symmetric random matrices over $\Z_p$ are distributed according to Delaunay's heuristics as their size goes to infinity.  Our strategy and proofs are closely analogous to those in \cite[Section~3]{BKLPR13}, to which we refer the reader for a beautiful treatment of similar results in which the duality pairing is replaced by a nondegenerate alternating form.

\begin{rem}
Theorem~\ref{thm:pparts} is used by the fifth author in \cite{Wood14}, in combination with a universality result
that says cokernels of much more general random symmetric matrices, including Laplacians of random graphs, 
are distributed in the same way as cokernels of Haar random symmetric matrices.
\end{rem}

Let $A$ be nonsingular symmetric $n \times n$ matrix with entries in $\Z_p$.  Then $A$ induces a natural symmetric bilinear pairing
\[
\pair_A: \Z_p^n \times \Z_p^n \rightarrow \Q_p/\Z_p,
\]
given by 
\[
\langle x, y \rangle_A = y^t A^{-1} x.
\]
Let $\Gamma$ be the cokernel of $A$.  If $x$ or $y$ is in the image of $A$, then $\langle x, y \rangle_A = 0$, so there is an induced symmetric pairing $\delta_A : \Gamma \times \Gamma \rightarrow \Q_p / \Z_p$. 
The image of $\delta_A$ is a subgroup of $\frac{1}{|\Gamma|} \Z_p/\Z_p$, which is naturally identified with $\frac{1}{|\Gamma|} \Z/\Z \subset \Q/\Z$.  The resulting map from $\Gamma$ to $\Hom(\Gamma, \Q/\Z)$ is injective, and hence an isomorphism.  In particular, $(\Gamma, \delta_A)$ is a finite abelian $p$-group with duality pairing.

The  goal of this section is to prove Theorem~\ref{thm:pparts}, which gives the distribution on the groups with duality pairing appearing as the cokernel of $A$, when $A$ is chosen randomly with respect to additive Haar measure on $\Sym_n(\Z_p)$.  More concretely, a random element $x \in \Z_p$ chosen with respect to Haar measure is given by independently choosing a random element of $\Z/p\Z$ for each digit of the $p$-adic expansion of $x$.  To choose a random $A \in \Sym_n(\Z_p)$ with respect to this measure, we simply choose a random element for each entry $a_{i,j}$ of $A$ with $i \le j$.

The structure of the argument is similar to the proof of \cite[Theorem~3.9]{BKLPR13}.  One key ingredient in the argument is the following characterization of pairs of matrices that determine the same pairings  $\pair: \Z_p^n \times \Z_p^n \rightarrow \Q_p/\Z_p$.  For a matrix $M$ with entries in $\Z_p$, we write $\overline M$ for the matrix with entries in $\Z/p\Z$ given by reducing the entries of $M$ modulo $p$.

\begin{lemma}[Lemma~3.2 of \cite{BKLPR13}]
\label{BKLemma}
Let $A$ and $M$ be nonsingular  $n \times n$ matrices with entries in $\Z_p$.  Then the pairings $\pair_A$ and $\pair_M$ from $\Z_p^n \times \Z_p^n$ to $\Q_p/\Z_p$ are the same if and only if
\begin{enumerate}
\item the matrix $A$ is equal to $M + MRM$, for some $R \in M_{n\times n}(\Z_p)$, and
\item the rank of $\overline A$ is equal to the rank of $\overline M$.
\end{enumerate}
\end{lemma}

We also use the following formula for counting pairings.  For an arbitrary symmetric bilinear pairing $\sqpair: \Z_p^n \times \Z_p^n \rightarrow \Q_p/\Z_p$, we write $\coker \sqpair$  for the finite abelian $p$-group
\[
\coker \sqpair = \Z_p^n / \{x \in \Z_p^n \ | \ [x,y] = 0 \mbox{ for all } y \in \Z_p^n \}.
\]
Note that $\coker \sqpair$ carries a canonical duality pairing, induced by $\sqpair$.

\begin{lemma}\label{lem:numpairings}
The number of symmetric bilinear pairings $\sqpair: \Z_p^n \times \Z_p^n \rightarrow \Q_p / \Z_p$ such that $\coker \sqpair$ with its canonical duality pairing is isomorphic to $(\Gamma, \delta)$ is
\[
\frac{\# \Gamma^n \cdot \prod_{j = n-r+1}^n (1 - p^{-j})} {\# \Aut(\Gamma, \delta)},
\]
where $r$ is the rank of $\Gamma$.
\end{lemma}

\begin{proof}
A symmetric bilinear pairing $\sqpair: \Z_p^n \times \Z_p^n \rightarrow \Q_p/\Z_p$ with a choice of an isomorphism $\coker \sqpair \xrightarrow{\sim} \Gamma$ respecting the pairing is equivalent to a surjection $\Z_p^n \rightarrow \Gamma$.  In each instance, there are $\# \Aut(\Gamma,\delta)$ choices of isomorphisms from $\coker \sqpair$ to $\Gamma$ that respect the pairing, so the number of distinct pairings with cokernel isomorphic to $(\Gamma, \delta)$ is 
\[
\#\Sur(\Z_p^n,\Gamma)/\# \Aut(\Gamma,\delta).
\]
Note that a homomorphism $\Z_p^n\ra \Gamma$ is surjective if and only if it is a surjection modulo $p$, by Nakayama's Lemma.  Therefore,
\[
\#\Sur(\Z_p^n,\Gamma)=\frac{\#\Gamma^n \cdot \prod_{i=0}^{r-1}(p^n-p^i)}{p^{nr}},
\]
and the lemma follows.
\end{proof}

The final key ingredient in our argument is the classification of symmetric bilinear forms over $\Q_p$ up to $\GL_n(\Z_p)$ equivalence.  When $p$ is odd, any symmetric bilinear form over $\Q_p$ is diagonalizable by a $p$-adic integral change of coordinates.  More precisely, if $M$ is a symmetric matrix with entries in $\Q_p$, then we can choose $H \in \GL_n(\Z_p)$ such that $HMH^t$ is diagonal \cite[Section~15.4.4]{ConwaySloane99}. The classification of symmetric bilinear forms over $\Q_2$ is more complicated, so we treat this case separately at the end of the section.

\begin{proof}[Proof of Theorem~\ref{thm:pparts} for odd $p$.]
Let $A$ be a symmetric $n \times n$ matrix with entries in $\Z_p$, chosen randomly with respect to additive Haar measure.   The singular matrices have Haar measure zero, so we may assume $A$ is nonsingular.   We want to determine the probability that the cokernel of $A$ with duality pairing $\delta_A$ is isomorphic to $(\Gamma, \delta)$.   Note that $\langle x, y \rangle_A$ is zero for all $y \in \Z_p^n$ if and only if $x \in A \Z_p^n$, so the cokernel of $A$ with its duality pairing is canonically isomorphic to $\coker \pair_A$ with its induced pairing.

We now compute the probability that $\pair_A$ is equal to a fixed symmetric bilinear pairing $\sqpair: \Z_p^n \times \Z_p^n \rightarrow \Q_p / \Z_p$ such that $\coker \sqpair$ with its induced pairing is isomorphic to $(\Gamma, \delta)$.  Since $A$ is chosen randomly with respect to Haar measure, this probability is invariant under a change of basis on $\Z_p^n$.  

Assume $p$ is odd, and let $N \in \Sym_n(\Q_p)$ be a symmetric matrix whose $(i,j)$ entry is a lift of $[e_i,e_j]$ to $\Q_p$.  By the classification of symmetric bilinear forms over $\Q_p$ up to $\GL_n(\Z_p)$ \cite[Section~15.4.4]{ConwaySloane99}, after a change of basis on $\Z_p^n$ we may assume $N$ is diagonal. 
 Possibly changing the lift to $\Q_p$, we may further ensure that $N$ has diagonal entries with valuations $-d_i$, where
\[
d_1=\dots =d_{n-r}=0 \mbox{ \ and \ } 1\leq d_{n-r+1}\leq\dots\leq d_n.
\]
We therefore assume that $N$ is in this form.  Let $M = N^{-1}$ and note that, by construction, we have $\sqpair = \pair_M$.
 
We have shown that the probability that $\pair_A = \sqpair$ is the same as the probability that $\pair_A = \pair_M$, where $M$ is a diagonal matrix whose $i$th diagonal entry has valuation $d_i$.  By Lemma~\ref{BKLemma}, the pairings $\pair_A$ and $\pair_M$ are equal if and only if $A = M + MRM$ for some $R \in M_{n \times n}(\Z_p)$ and $\rank (\overline A) = \rank (\overline M)$.  We now determine the probability that $\pair_A = \pair_M$.

The condition that $A$ be of the form $M + MRM$ is equivalent to requiring that $M^{-1}(A-M)M^{-1}$ is in $M_{n\times n}(\Z_p).$  Given $M$, this is equivalent to the entries $a_{i,j}$ of $A$ satisfying certain divisibility conditions. 
Let $p^{d_i}u_i$  be the $i$th diagonal entry of $M$, with $u_i\in\Z_p^*$.
For $i<j$, the $(i,j)$ entry of $M^{-1}(A-M)M^{-1}$ is in $\Z_p$ if and only if $p^{d_i+d_j}\ \mid a_{i,j}$, and the $(i,i)$ entry is in $\Z_p$ if and only if $p^{2d_i} \mid (a_{i,i} - p^{d_i}u_i)$.  Therefore, for all $i \leq j$, the condition $A = M + MRM$ is equivalent to fixing the first $d_i+ d_j$ digits in the $p$-adic expansion of $a_{i,j}$.  Note, in particular, that when $d_i = d_j = 0$, there is no condition at all on the entry $a_{i,j}$.
  The probability that a random symmetric matrix satisfies these conditions is then
\[
\prod_{1\le i \le j \le n} p^{-(d_i + d_j)} = \prod_{i=1}^n p^{-(n+1) d_i} = \frac{1}{\# \Gamma^{n+1}}.
\]

Furthermore, given the above valuation conditions, we see that $\overline A$ is zero outside the upper left  $(n-r) \times (n-r)$ minor. The matrix $\overline M$ is also zero outside the upper left  $(n-r) \times (n-r)$ minor, and this minor is a diagonal matrix with non-zero entries.  In particular, $\rank (\overline M)=n-r$.
The condition that $\rank( \overline A) = \rank (\overline M)$ is independent from the divisibility conditions, and holds with probability equal to the proportion of invertible matrices in $\Sym_{n-r}(\F_p)$.  In particular, the probability that $\pair_A = \pair_M$, and hence the probability that $\pair_A = \sqpair$, is 
\[
\frac{\# (\GL_{n-r}(\F_p) \cap \Sym_{n-r}(\F_p))}{\# \Sym_{n-r}(\F_p) \cdot \# \Gamma^{n+1}}.
\]
Note that this probability depends only on $\# \Gamma$, and is independent of all other choices.

Lemma~\ref{lem:numpairings} gives the number of symmetric bilinear pairings $\sqpair: \Z_p^n \times \Z_p^n \rightarrow \Q_p / \Z_p$ such that $\coker \sqpair$ with its induced duality pairing is isomorphic to $(\Gamma, \delta)$.  We conclude that the probability that the cokernel of $A$ with its duality pairing is isomorphic to $(\Gamma, \delta)$ is the product
\begin{equation} \label{eqn:probability}
\frac{\# (\GL_{n-r}(\F_p) \cap \Sym_{n-r}(\F_p))}{\# \Sym_{n-r}(\F_p) \cdot \# \Gamma^{n+1}} \cdot \frac{\# \Gamma^n \cdot \prod_{j = n-r+1}^n (1 - p^{-j})} {\# \Aut(\Gamma, \delta)} .
\end{equation}
By \cite[Theorem~2]{MacWilliams69}, the number of invertible matrices in $\Sym_{k}(\F_p)$ is
\[
p^{\binom{k+1}{2}} \prod_{j=1}^{\lceil \frac{k}{2} \rceil} (1-p^{1-2j}).
\]
Therefore, the expression in (\ref{eqn:probability}) can be rewritten as
\[
\frac{\prod_{i=1}^{\lceil(n-r)/2\rceil} (1-p^{1-2i}) \cdot \prod_{j=n-r+1}^n (1-p^{-j}) }{\# \Gamma \cdot \# \Aut(\Gamma,\delta)},
\]
and the theorem follows.
\end{proof}

\bigskip

We now return to the case $p=2$ and show that, for a given symmetric pairing $\sqpair:\Z_p^n\times\Z_p^n \ra \Q_p/\Z_p$, the probability that $\pair_A=\sqpair$  is again
\[
\frac{\# (\GL_{n-r}(\F_p) \cap \Sym_{n-r}(\F_p))}{\# \Sym_{n-r}(\F_p) \cdot \# \Gamma^{n+1}}.
\]
 
\noindent The classification of symmetric bilinear forms over $\Q_2$ is due to Wall \cite{Wall64}. It is given in the following form in \cite[Theorem 2, Section~15.4.4]{ConwaySloane99}.
\begin{thm}[\!\!\cite{ConwaySloane99}]\label{CS}
Suppose that $A \in \Sym_n(\Z_2)$.  Then there exists a matrix $H \in \GL_n(\Z_2)$ such that $H A H^t$ is a block diagonal matrix consisting of: 
\begin{enumerate}
\item diagonal blocks $u_i 2^{d_i}$, where each $d_i \ge 0$ and $u_i$ is a unit in $\Z_2$, 
\item $2\times 2$ blocks of the form $2^{e_i} \left(\begin{smallmatrix} a & b \\ b & c \end{smallmatrix} \right)$, where $e_i \ge 0,\ b$ is a unit in $\Z_2$ and $a,c \in 2\Z_2$.
\end{enumerate}
\end{thm}

\begin{proof}[Proof of Theorem~\ref{thm:pparts} for $p = 2$]
As in the case for odd $p$, we let $A$ be a symmetric $n\times n$ matrix with entries in $\Z_2$ chosen randomly with respect to additive Haar measure.  We will again consider the probability that the cokernel of $A$ with duality pairing $\delta_A$ is isomorphic to $(\Gamma,\delta)$.  
We compute the probability that $\pair_A$ is equal to a fixed symmetric bilinear pairing $\sqpair: \Z_2^n \times \Z_2^n \rightarrow \Q_2 / \Z_2$ such that $\coker \sqpair$ with its induced pairing is isomorphic to $(\Gamma, \delta)$.  Just as for odd $p$, since $A$ is chosen randomly with respect to Haar measure, this probability is invariant under a change of basis on $\Z_2^n$.  

Let $N \in \Sym_n(\Q_2)$ be a symmetric matrix whose $(i,j)$ entry is a lift of $[e_i,e_j]$ to $\Q_2$.  We first note that there exists an element $2^m \in \Z_2$ such that $2^m N \in \Sym_n(\Z_2)$. It is no longer true that we can change basis so that $2^m N$ is diagonal, but Theorem \ref{CS} shows that there exists $H\in \GL_2(\Z_2)$ such that $HNH^t$ is a block diagonal matrix with entries in $\Q_2$ that has two types of blocks, diagonal blocks given by a unit times $2^{-d_i}$, and $2\times 2$ blocks of the form $2^{-e_i} \left(\begin{smallmatrix} a_i & b_i \\ b_i & c_i \end{smallmatrix} \right)$, where $b_i \in \Z_2^*$ and $a_i,c_i \in 2\Z_2$.  By possibly changing the lift to $\Q_2$, we can suppose that each $d_i, e_i \ge 0$.  Since $a_i c_i  -b_i^2 \in \Z_2^*$, the inverse of a $2\times 2$ block of this form is a unit times $2^{e_i} \left(\begin{smallmatrix} c_i & -b_i \\ -b_i & a_i \end{smallmatrix}\right)$.  We note that $-1\in \Z_2^*$ and can therefore write this block as a unit times $2^{e_i} \left(\begin{smallmatrix} c_i & b_i \\ b_i & a_i \end{smallmatrix}\right)$ where $b_i \in \Z_2^*$ and $a,c \in 2\Z_2$.  We say that $e_i$ is the valuation of this block.  Therefore, $M = (HNH^t)^{-1}$ has entries in $\Z_2$ and by construction we have $\sqpair = \pair_M$.

As above, Lemma~\ref{BKLemma} implies that $\pair_A$ and $\pair_M$ are equal if and only if $A = M + MRM$ for some $R \in M_{n \times n}(\Z_2)$ and $\rank (\overline A) = \rank (\overline M)$.  The condition that $A$ be of the form $M + MRM$ is equivalent to $M^{-1} (A-M) M^{-1} \in M_{n\times n}(\Z_2)$, and gives divisibility conditions on the entries $a_{i,j}$ of $A$.  We determine these conditions by explicitly computing the relevant entries of $M+MRM$. For odd $p$ we could suppose that $M$ was a diagonal matrix such that the diagonal entry in row $i$ had valuation $d_i$.  There were two kinds of divisibility results, one for $a_{i,j}$ with $i<j$, and another for $a_{i,i}$. Combining these constraints, we saw that there exists a matrix $R \in M_{n \times n}(\Z_p)$ such that $A = M + MRM$ if and only if for each $i\le j$ the first $d_i+d_j$ entries in the $p$-adic expansion of $a_{i,j}$ are given by particular values determined by $M$.  For $p=2$ there are more kinds of constraints to consider. However, we will show that the same statement is true, that there exists an $R \in M_{n \times n}(\Z_2)$ such that $A = M + MRM$ if and only if for each $i\le j$ the first $d$ entries in the $p$-adic expansion of $a_{i,j}$ are given by a specific set of values determined by $M$, where $d$ is the sum of the valuations of the blocks of $M$ in rows $i$ and $j$.

After permuting coordinates we may suppose that the initial rows of $M$ have only diagonal nonzero entries $u_i 2^{d_i}$ where each $u_i \in \Z_2^*$ and $0 \le d_0 \le d_1 \le \cdots \le d_{k_1}$, followed by a set of $2\times 2$ diagonal blocks $2^{e_i} \left(\begin{smallmatrix} a_i & b_i \\ b_i & c_i \end{smallmatrix} \right)$, where $b_i \in \Z_2^*,\ a_i,c_i \in 2\Z_2$, and $0\le e_1\le \cdots \le e_{k_2}$.  By invariance of Haar measure, we see that the probability that $\pair_A = \sqpair$ is the same as the probability that $\pair_A = \pair_M$, where $M$ is a block diagonal matrix of this form.

Let $R \in M_{n\times n}(\Z_2)$ have entries $r_{i,j}$.  We explicitly compute the entries of the matrix $MRM$.  Row $i$ of $M$ can correspond either to a diagonal $1\times 1$ block, the first row of a $2\times 2$ block of the type described in the previous paragraph, or the second row of such a block.

Suppose that rows $i$ and $j$ of $M$ each correspond to $1\times 1$ blocks with valuations $d_i$ and $d_j$, respectively.  The $(i,j)$ entry of $MRM$ is then a unit times $2^{d_i + d_j} r_{i,j}$. An appropriate choice of $r_{i,j} \in \Z_2$ shows that this entry can be any element of $2^{d_i+d_j} \Z_2$.  Now suppose that row $i$ corresponds to a $1\times 1$ block with valuation $d_i$ and that rows $j$ and $j+1$ correspond to the $2\times 2$ block $2^{e_j} \cdot \left(\begin{smallmatrix} a & b\\ b & c \end{smallmatrix} \right)$.   Then the $(i,j)$ entry of $MRM$ is a unit times
\[
2^{d_i+e_j} \left( r_{i,j} a + r_{i,j+1} b \right) 
\]
and the $(i,j+1)$ entry is a unit times
\[
2^{d_i+e_j} \left( r_{i,j} b + r_{i,j+1} c \right).
\]
Since $b \in \Z_2^*$, an appropriate choice of $r_{i,j+1}$ shows that this first entry can be any element of $2^{d_i + e_j} \Z_2$, and similarly, an appropriate choice of $r_{i,j}$ shows that the second entry can be any element of $2^{d_i + e_j} \Z_2$.

Now suppose that rows $i$ and $i+1$ correspond to the $2\times 2$ block $2^{e_i} \cdot \left(\begin{smallmatrix} a_i & b_i \\ b_i & c_i \end{smallmatrix} \right)$, and rows $j$ and $j+1$ correspond to the $2\times 2$ block $2^{e_j} \cdot \left(\begin{smallmatrix} a_j & b_j \\ b_j & c_j \end{smallmatrix} \right)$.  Then the four entries, $(i,j), (i+1,j),(i,j+1),(i+1,j+1)$, of $MRM$ are given by
\begin{eqnarray*}
2^{e_i+e_j} & \left( a_i a_j r_{i,j} + b_i a_j r_{i+1,j} + a_i b_j r_{i,j+1} + b_i b_j r_{i+1,j+1}\right) & \\
2^{e_i+e_j} & \left( a_i b_j r_{i,j} + b_i b_j r_{i+1,j} + a_i c_j r_{i,j+1} + b_i c_j r_{i+1,j+1}\right) & \\
2^{e_i+e_j} & \left( b_i a_j r_{i,j} + c_i b_j r_{i+1,j} + b_i b_j r_{i,j+1} + c_i b_j r_{i+1,j+1}\right) & \\
2^{e_i+e_j} & \left( b_i b_j r_{i,j} + c_i b_j r_{i+1,j} + b_i c_j r_{i,j+1} + c_i c_j r_{i+1,j+1}\right), &
\end{eqnarray*}
respectively.  Since $b_i b_j \in \Z_2^*$, appropriate choices of $r_{i+1,j+1}, r_{i+1,j}, r_{i,j+1}$ and $r_{i,j}$, respectively, show that these four entries can each be any elements of $2^{e_i+e_j} \Z_2$.  Note that we are allowing the case where $i=j$ and these two blocks are identical.

We conclude from this analysis that by an appropriate choice of $R,\ MRM$ can be any matrix where the valuation of the $(i,j)$ entry is at least the sum of the valuations of the blocks corresponding to rows $i$ and $j$.  We see that this fixes the initial $p$-adic digits of every entry of the matrix $A-M$.  The probability that a matrix $A$ satisfies all of these conditions is given by the product of the probabilities for individual entries.  Suppose that row $k$ of $A$ corresponds to a block with valuation $d$.  Considering the divisibility constraints from all entries $a_{k,j}$ in this row with $k\le j$ and all $a_{l,k}$ with $l<k$ contributes a factor of $2^{-(n+1)d}$ to this probability.  Taking the product over all rows gives the total probability that there exists an $R \in M_{n \times n}(\Z_2)$ such that $A = M + MRM$.

The determinant of $M$ is the product of the determinants of the diagonal blocks.  For $1\times 1$ diagonal blocks the determinant is a unit times $2^{d_i}$.  A $2\times 2$ block has determinant equal to a unit times $2^{2e_i}$.  As in the case for odd $p,\ \# \Gamma = 2^v$, where $v$ is the $2$-adic valuation of $\det(M)$.

By permuting rows and columns of $M$, we may suppose that $\overline{M}$ is zero outside of the upper left $(n-r)\times (n-r)$ minor, and is a block diagonal matrix of rank $n-r$. The divisibility conditions also imply that $\overline{A}$ is zero outside of its upper left $(n-r)\times (n-r)$ minor.  The condition that $\rank(\overline{A}) = \rank(\overline{M})$ is again independent of the divisibility conditions. Taking the product over the rows of $A$ shows that the probability that $\pair_A=\sqpair$ is
\[
\frac{\# (\GL_{n-r}(\F_2) \cap \Sym_{n-r}(\F_2))}{\# \Sym_{n-r}(\F_2)) \cdot \# \Gamma^{n+1}}.
\]
The rest of the argument is the same as for odd $p$.
\end{proof}

\subsection{A generalization of Theorem~\ref{thm:pparts} for matrices with row and column sums equal to zero.}

The combinatorial Laplacian of a graph is a symmetric matrix with each row and column sum equal to zero.  We now generalize Theorem~\ref{thm:pparts} to random matrices satisfying this condition.  Let $\Sym_n^0 \subset \Sym_n$ be the symmetric matrices with each row and column sum equal to $0$.  We adapt the construction of a finite abelian group with duality pairing associated to a symmetric nonsingular matrix to the case where the matrix may be singular, as follows.

Let $A \in \Sym_n(\Z_p)$ be a possibly singular matrix, and define
\[
L_A = \Z_p^n \cap A\Q_p^n.
\]
Then $A$ determines a natural symmetric pairing $L_A \times L_A \rightarrow \Q_p/\Z_p$,  where if $x=\frac{1}{m} Az$ for $z\in \Z_p^n$ and $m\in \Z_p$, then
$(x,y)\mapsto \frac{1}{m} z^ty$.
We define the finite cokernel of $A$ to be $\Gamma =L_A/ A\Z_p^n$.  This is exactly the torsion subgroup of the cokernel of $A$.  The above pairing descends to 
\[
\Gamma \times \Gamma \rightarrow \Q_p/\Z_p,
\]
which, as in the nonsingular case discussed at the beginning of the section, is a duality pairing on $\Gamma$.

\begin{thm} \label{thm:zerosum}
Let $\Gamma$ be a finite abelian $p$-group and $\delta$ a duality pairing on $\Gamma$.  Choose an $A\in \Sym_n^0(\Z_p)$ randomly with respect to additive Haar measure.  Let $\mu_n(\Gamma,\delta)$ be the probability that the finite cokernel of $A$ with its duality pairing is isomorphic to $(\Gamma,\delta)$.  Let $r$ be the $p$-rank $\dim_{\Z/p\Z} \Gamma/p\Gamma$.  Then 
\[
\mu_n(\Gamma,\delta) = \frac{\prod_{j=n-r}^{n-1} (1-p^{-j}) \prod_{i=1}^{\lceil(n-1-r)/2\rceil} (1-p^{1-2i})}{\# \Gamma \cdot \# \Aut(\Gamma,\delta)}.
\]
In particular,
\[
\lim_{n\rightarrow \infty} \mu_n(\Gamma,\delta)  = \frac{\prod_{i=1}^\infty (1-p^{1-2i})}{\# \Gamma \cdot \# \Aut(\Gamma,\delta)}.
\]
\end{thm}

\begin{proof}
Given a matrix $A\in \Sym_n^0(\Z_p)$, deleting the last row and column gives a matrix $A'$ in $\Sym_{n-1}(\Z_p)$.  Clearly $A'$ determines $A$ as well, and the correspondence respects Haar measure.

Let $Z_0$ be the subset of $\Z_p^n$ where the coordinates sum to $0$.  We have that $L_A\subseteq Z_0$ and if equality does not hold then some subset of the $n-1$ columns of $A'$ are linearly dependent over $\Q_p$.   In particular $A'$ has determinant $0$, which happens with probability $0$.  Therefore, $L_A = Z_0$ with probability $1$, which we assume from now on.

Let $e_i$ be the standard basis for $\Z_p^n$.
We can explicitly check that the matrix $A'$ gives the same pairing on $\Z_p^{n-1}$ that $A$ gives on $Z_0$ by using the basis $e_i - e_n$ for $Z_0$.  This theorem now follows from Theorem \ref{thm:pparts}.
\end{proof}

\section{Computing Averages}\label{S:averages}
In this section, we make theoretical computations of the exact values that are predicted by Conjecture~\ref{C:basic}.
These are (unconditional) results of group theory about the values taken by the right-hand side of Conjecture~\ref{C:basic} for several important functions $F$. 
Recall that $\A(m)$ is the set of all isomorphism classes of pairs $(\Gamma,\delta)$ where $\Gamma$ is an abelian group of order $m$ and $\delta$ is a duality pairing on $\Gamma$.

\subsection{The normalizing constant}
Here we compute a constant that will be important in our later computations, as it arises when computing the denominator of our averages.

\begin{prop}\label{P:normalizing}
For any prime $p$, we have
\[
\sum_{m = 0}^{\infty} \sum_{(\Gamma,\delta)\in \A(p^m)} \frac{1}{\# \Gamma\cdot \# \Aut(\Gamma,\delta)} = \prod_{i=1}^{\infty} (1-p^{1-2i})^{-1}.
\]
\end{prop}

\begin{proof}
Let $(\Gamma,\delta)$ be a finite abelian $p$-group with duality pairing.  From Theorem~\ref{thm:pparts}, we know $\mu_n(\Gamma,\delta)$, the probability that a random matrix in $\Sym_n(\Z_p)$ has cokernel isomorphic to $(\Gamma,\delta)$.  For each $n$ we have that 
\[
\sum_{m = 0}^{\infty} \, \sum_{(\Gamma,\delta) \in \A(p^m)} \mu_n(\Gamma,\delta) = 1.
\]
Recall that
\[
\lim_{n\rightarrow \infty} \mu_n(\Gamma,\delta) = \frac{\prod_{i=1}^{\infty} (1-p^{1-2i})}{\# \Gamma\cdot \# \Aut(\Gamma,\delta)}.
\]

Let $\mu_{\max}(\Gamma,\delta)=\max_{n} \mu_n(\Gamma,\delta)$.
It is not hard to see that there is an absolute constant $c$ such that for any $n$ at least the rank of $\Gamma$, we have $\mu_{\max}(\Gamma,\delta)\leq c \mu_n(\Gamma,\delta),$ e.g. take $c=\prod_{i\geq 1} (1-2^{-1})^{-2}$. So for any $k$
\begin{align*}
\sum_{m=0}^\infty \sum_{(\Gamma,\delta) \in \A(p^m) \atop r_p(\Gamma) \le k} \mu_{max}(\Gamma,\delta) 
&\le \sum_{m=0}^\infty \sum_{(\Gamma,\delta) \in \A(p^m) \atop r_p(\Gamma) \le k} c \mu_{k}(\Gamma,\delta) \\
&\le \sum_{m=0}^\infty \sum_{(\Gamma,\delta) \in \A(p^m)} c \mu_m(\Gamma,\delta)\\
&\leq c.
\end{align*}
Taking the limit as $k$ goes to infinity shows that
$$
\sum_{m=0}^\infty \sum_{(\Gamma,\delta) \in \A(p^m)} \mu_{max}(\Gamma,\delta) 
$$
converges.
Then by the Lesbesgue Dominated Convergence Theorem, we have
\[
\lim_{n\rightarrow \infty} \sum_{m=0}^\infty \sum_{(\Gamma,\delta) \in \A(p^m)} \mu_n(\Gamma,\delta) =
\sum_{m=0}^\infty \sum_{(\Gamma,\delta) \in \A(p^m)} \lim_{n\rightarrow \infty} \mu_n(\Gamma,\delta) ,
\]
as desired.
\end{proof}

\subsection{The probability that  $\Gamma$ has trivial or cyclic $p$-part}
For each prime $p$, we define
\[
C_p = \prod_{i=1}^{\infty} (1-p^{1-2i}).
\]
We have the following, using the measure $\mu$ defined in the introduction.
\begin{prop}
Let $F_{p-trivial}(\Gamma) = 1$ if $r_p(\Gamma) =  0$, and $0$  otherwise. Then
\[
\int_{(F,\delta)\in \mathcal{A}_p}F(\Gamma)d\mu= C_p.
\]
\end{prop}

\begin{proof}
The measure $\mu$ of the trivial group is $C_p$.
\end{proof}

\noindent This shows for example, that the probability that a group obeying our conjectures has trivial $2$-part is a little over $.4194$, and the probability that it has trivial $17$-part is a little over $.9409$.

\begin{prop} \label{P:cyclic}
Let $F_{cyclic}(\Gamma)  = 1$ if $\Gamma$ is cyclic, and $0$ otherwise.  Then 
\[
\int_{(\Gamma,\delta)\in \mathcal{A}_p}F(\Gamma)d\mu
 = \frac{C_p}{1-p^{-1}}.
\]
\end{prop}

\begin{proof}
By definition, we have
\[
\int_{(F,\delta)\in \mathcal{A}_p}F(\Gamma)d\mu
 =C_p \sum_{m=0}^\infty \sum_{\substack{(\Gamma,\delta) \in \A(p^m)\\ \Gamma \textrm{ cyclic}}} \frac{1}{\# \Gamma\cdot \# \Aut(\Gamma,\delta)}.
\]
\noindent For each $m$ we show that 
\begin{equation}\label{E:tocompute}
\sum_{\substack{(\Gamma,\delta) \in \A(p^m)\\ \Gamma \textrm{ cyclic}}} \frac{1}{\# \Gamma\cdot \# \Aut(\Gamma,\delta)} = \frac{1}{p^m}.
\end{equation}
There is only one cyclic group of given order, so we only need to take the sum over the different pairings.  We claim that the number of isomorphism classes of $(\Z/p^m\Z,\delta)$ is equal to the size of $\Aut(\Z/p^m\Z,\delta)$ for each duality pairing $\delta$.  

A duality pairing $\delta$ on the cyclic group $\Z/p^m\Z$ is determined by its value on $(1,1)$.  Changing basis for $\Z/p^m\Z$, replacing $1$ by a generator $u \in (\Z/p^m\Z)^*$, multiplies this value by a factor of $u^2$.  Therefore, the isomorphism classes of duality pairings on $\Z/p^m\Z$ correspond naturally to the cosets in $(\Z/p^m\Z)^*$ modulo squares, and the number of automorphisms of each pairing is the number of square roots of $1$.  These are equal, since one is the index of the image and one is the size of the kernel for the endomorphism $u \mapsto u^2$.  This proves the claim, and the lemma follows. 
\end{proof}

\subsection{Moments}

The Cohen-Lenstra distribution is a measure on abelian $p$-groups such that the expected number of surjections to any finite abelian $p$-group $\Gamma$ is $1$.  These are often called the \emph{moments} of the distribution, and we now explain why.  
Let $\Gamma \cong \prod_{i=1}^r \Z/p^{f_i}\Z$ with $f_1\le f_2 \le \cdots \le f_r$ and
$\Gamma' \cong \prod_{i=1}^r \Z/p^{e_i}\Z$ with $e_1\le e_2 \le \cdots \le e_r$.
Let $f'_1\geq \cdots$ be the transpose of the partition $f_r\geq \dots \geq f_1$, and 
let $e'_1\geq \cdots$ be the transpose of the partition $e_r\geq \dots \geq e_1$.  Note that the $f_i'$  are a complete set of invariants for the finite abelian $p$-group $\Gamma$ (and thus so are  $p^{f_i'}$).  It is a standard fact that $\# \Hom ( \Gamma, \Gamma')=p^{\sum_i f_i'e_i'}$.
Thus, for any measure $\nu$ on finite abelian $p$-groups
$$
\int_{\Gamma} \#\Hom(\Gamma,\Gamma') d\nu=\int_{\Gamma} \prod_i (p^{f_i'})^{e_i'} d\nu
$$
is the $e'_1,e'_2,\dots$ mixed moment (in the usual sense) of the variables $p^{f_1'}, p^{f_2'}, \dots$.  The averages 
$$
\int_{\Gamma} \#\Hom(\Gamma,\Gamma') d\nu
$$
for all subgroups $\Gamma'$ of an abelian $p$-group $A$ are related to the 
averages 
$$
\int_{\Gamma} \#\Sur(\Gamma,\Gamma') d\nu
$$
for all subgroups $\Gamma'$ of $A$ by an upper-triangular linear transformation (only depending on $A$).  In other words, the $\Hom$ averages, taken together, are equivalent data to the $\Sur$ averages.  However, in practice, the $\Sur$ averages often seem to capture more basic algebraic data, and the $\Hom$ averages are usually best understood as the sum of the $\Sur$ averages over subgroups (e.g. in the case of the Cohen-Lenstra measure all the $\Sur$ averages are $1$).   Hence the $\Sur$ averages are often studied and called the moments.

We now turn to computing moments of measure $\mu$ defined in the introduction.

\begin{thm}\label{T:momint}
Let $\Gamma' \cong \prod_{i=1}^r \Z/p^{e_i}\Z$ with $e_1\le e_2 \le \cdots \le e_r$.
Then 
$$\int_{(\Gamma,\delta)\in\A_p} \#\Sur(\Gamma,\Gamma') d\mu =p^{(r-1)e_1 + (r-2)e_2+\cdots + e_{r-1}}.$$
\end{thm}

\noindent We prove this result after first establishing the following analogous result in the random matrix case.
\begin{thm}\label{T:mom}
Suppose $\Gamma' \cong \prod_{i=1}^r \Z/p^{e_i}\Z$ with $e_1\le e_2 \le \cdots \le e_r$.  Let $A$ be a random matrix in $\Sym_n(\Z_p)$ with respect to (additive) Haar measure.   As $n$ goes to infinity, the expected number of surjections from the cokernel of $A$ to $\Gamma'$ approaches $p^{(r-1)e_1 + (r-2)e_2+\cdots + e_{r-1}}$.
\end{thm}

\begin{proof}
Let $A \in \Sym_n(\Z_p)$ be chosen randomly with respect to Haar measure.  With probability $1,\ \det(A) \neq 0$, so we assume that from now on.  There are $p^{(e_1+e_2+\cdots + e_r) n}$ distinct maps $\Z_p^n \rightarrow \Gamma'$.  As $n$ goes to infinity, the probability that such a map is a surjection goes to $1$.  We choose such a map at random and compute the probability that it factors through the cokernel of $A$.  

The kernel of a surjection from $\Z_p^n$ to $\Gamma'$ is given by the column space of a matrix $B\in M_{n\times n}(\Z_p)$.  Given $B$, we determine the probability that $A \Z_p^n \subset B \Z_p^n$.  This probability is unchanged by a change of basis on $\Z_p^n$.  We put $B$ into Smith normal form, first multiplying on the right and then on the left by matrices in $\GL_n(\Z_p)$.  We choose $G, H \in \GL_n(\Z_p)$ so that $GBH$ is diagonal. Since $G^t \Z_p^n = H \Z_p^n = \Z_p^n$, the probability that $A\Z_p^n \subset B\Z_p^n$ is equal to the probability that $G A G^t \Z_p^n \subset G B H \Z_p^n$.  By the properties of Haar measure, $GAG^t$ is drawn from the same distribution as $A$. 

We now need only compute the probability that $A \in D \Z_p^n$, where $D$ is a diagonal matrix with $r$ diagonal entries of positive valuation, $u_i p^{e_i}$ where each $u_i$ is a unit in $\Z_p$ and $1\le e_1 \le e_2 \le \cdots \le e_r$.  Note that $\Z_p^n/D\Z_p^n \cong \Gamma'$.

The condition that $A \in D \Z_p^n$ can now be phrased in terms of divisibility conditions on the entries of $A$ that are determined by the $e_i$.  Suppose that the $k$th row of $D$ has diagonal entry equal to a unit times $p^{e_i}$.  Then this condition implies that every entry of the $k$th row and column of $A$ must have valuation at least $e_i$.  

We count the number of independent entries affected and see that a random symmetric matrix satisfies all of these conditions with probability $$p^{-e_r n - e_{r-1}(n-1)- \cdots - e_1 (n-(r-1))}.$$  We multiply by the $p^{(e_1+e_2+\cdots + e_r) n}$ maps $\Z_p^n \rightarrow \Gamma'$, which are almost all surjections as $n$ goes to infinity.  This shows that the expected number of surjections is $p^{(r-1)e_1+(r-2) e_2+\cdots + e_{r-1}}$.
\end{proof}

For a partition $\lambda$ given by $\lambda_1\geq \lambda_2\geq \dots$, we let $\Gamma'_\lambda =\oplus_i \Z/p^{\lambda_i}\Z$.
Let $\lambda'$ denote the transpose of $\lambda$.
Note that $\sum_i (i-1)\lambda_i$ is the sum over boxes in the partition diagram of $\lambda$ of $i-1$, where $i$ in the row the box appears in.  Summing by column, we obtain
$\sum_i (i-1)\lambda_i=
\sum_j \frac{\lambda'_j(\lambda'_j-1)}{2}.
$ 
So we have, when $n-1\geq \lambda'_1$, the expected number of surjections from the cokernel of $A\in \Sym_n(\Z_p)$ to $\Gamma'_\lambda$ is given by 
\begin{equation}\label{E:ESur}
 p^{\sum_i (i-1)\lambda_i}=p^{\sum_j \frac{\lambda'_j(\lambda'_j-1)}{2}}.
\end{equation}

\begin{proof}[Proof of Theorem~\ref{T:momint}] 

In Theorem~\ref{T:mom}, 
we have computed 
\[
\lim_{n\rightarrow \infty} \sum_{m=0}^\infty \sum_{(\Gamma,\delta) \in \A(p^m)} \mu_n(\Gamma,\delta) \#\Sur(\Gamma,\Gamma'),
\] and so it remains to show that
\[
\lim_{n\rightarrow\infty} \sum_{m=0}^\infty \sum_{(\Gamma,\delta) \in \A(p^m)}  \mu_n(\Gamma,\delta) \#\Sur(\Gamma,\Gamma')=
 \sum_{m=0}^\infty \sum_{(\Gamma,\delta) \in \A(p^m)} \lim_{n\rightarrow \infty}  \mu_n(\Gamma,\delta) \#\Sur(\Gamma,\Gamma').
\]
We have 
\[
\mu_n(\Gamma,\delta) = \frac{\prod_{j=n-r+1}^n (1-p^{-j}) \prod_{i=1}^{\lceil(n-r)/2\rceil} (1-p^{1-2i})}{\# \Gamma \cdot \# \Aut(\Gamma,\delta)},
\]
where $r$ is the rank of $\Gamma$.  

We closely follow the proof of Proposition~\ref{P:normalizing}.  Let $\mu_{\max}(\Gamma,\delta):=\max_{n} \mu_n(\Gamma,\delta)$.
It is not hard to see that there is an absolute constant $c$ such that for any $n$ at least the rank of $\Gamma$, we have $\mu_{\max}(\Gamma,\delta)\leq c \mu_n(\Gamma,\delta),$ e.g. take $c=\prod_{i\geq 1} (1-2^{-1})^{-2}$.
So   for any $k$
\begin{align*}
\sum_{m=0}^\infty \sum_{(\Gamma,\delta) \in \A(p^m) \atop r_p(\Gamma) \le k} \mu_{max}(\Gamma,\delta) \#\Sur(\Gamma,\Gamma')
&\le \sum_{m=0}^\infty \sum_{(\Gamma,\delta) \in \A(p^m) \atop r_p(\Gamma) \le k} c \mu_m(\Gamma,\delta)\#\Sur(\Gamma,\Gamma')\\
&\le \sum_{m=0}^\infty \sum_{(\Gamma,\delta) \in \A(p^m)} c \mu_m(\Gamma,\delta)\#\Sur(\Gamma,\Gamma').
\end{align*}
Taking the limit as $k\rightarrow\infty$, we see the last term is bounded by Theorem~\ref{T:mom}.
So we have that 
$$
\sum_{m=0}^\infty \sum_{(\Gamma,\delta) \in \A(p^m)} \mu_{max}(\Gamma,\delta) \#\Sur(\Gamma,\Gamma')
$$
converges.
Then by the Lesbesgue Dominated Convergence Theorem, we have
\[
\lim_{n\rightarrow \infty} \sum_{m=0}^\infty \sum_{(\Gamma,\delta) \in \A(p^m)} \mu_n(\Gamma,\delta) \#\Sur(\Gamma,\Gamma') 
\]
is equal to 
\[
 \sum_{m=0}^\infty \sum_{(\Gamma,\delta) \in \A(p^m)} \lim_{n\rightarrow \infty} \mu_n(\Gamma,\delta) \#\Sur(\Gamma,\Gamma'),
\]
as desired.
\end{proof}

We now use this result to deduce the expectation of $r_p(A)$.  We recall the definition of the Gaussian binomial coefficient.  Let 
\[
\binom{k}{j}_p:= \prod_{i=0}^{j-1} \frac{p^k-p^i}{p^j-p^i}.
\]
This counts the number of $j$ dimensional subspaces of $\left(\Z/p\Z\right)^k$.  Note that $\binom{k}{0}_p = 1$ for any $k$ and $p$ since both the numerator and denominator consist of the empty product. 

\begin{thm}
 For a finite abelian $p$-group $\Gamma$, let $r_p(\Gamma)$ denote the $p$-rank of $\Gamma$, so $p^{r_p(\Gamma)}$ is the size of $\Gamma / p\Gamma$.  We have
\[
\int_{(\Gamma,\delta)\in\A_p} p^{k\cdot r_p(\Gamma)} d\mu = \prod_{j=0}^{k-1} (p^j+1).
\]
\end{thm}

\begin{proof}
We show that the left hand side of the statement is equal to
\[
\sum_{j=0}^{k} p^{j(j-1)/2} \binom{k}{j}_p.
\]
Applying the $q$-binomial theorem, formula (1.87) in \cite{Stanley12}, with $x=1$ completes the proof.

Let $\Gamma' = \left(\Z/p\Z\right)^k$.  The number of maps to $\Gamma'$ is $p^{k\cdot r_p(A)}$.  Such a map surjects onto a subgroup of $\Gamma'$ of size $p^j$ for some $j \in [0,k]$.  By the above theorem, the expected number of surjections onto a particular subgroup isomorphic to $\left(\Z/p\Z\right)^j$ is $p^{j(j-1)/2}$.  The number of subgroups of $\Gamma'$ isomorphic to $\left(\Z/p\Z\right)^j$ is equal to $\binom{k}{j}_p$.
\end{proof}

The expectation $p^{k\cdot r_p(A)}$ satisfies a particularly nice form, which suggests that the $p$-rank of the Jacobian of a random graph also satisfies a nice distribution.  This distribution has been determined by the fifth author as Corollary 9.4 of \cite{Wood14}.

\section{Empirical evidence}\label{S:Data}

Here we present some computational evidence for the conjectures stated in the introduction.  The code used to generate the data is available as part of the arxiv version of this paper \cite{arxiv}, after the {\textbackslash}end{\textbraceleft}document{\textbraceright} line in the source file.

\subsection{The probability that a random graph is cyclic}

We computed the Jacobians of $10^6$ connected random graphs with $n$ vertices and edge probability $q$, for $n \in \{ 15, 30, 45, 60\}$ and $q \in \{.3, .5, .7\}$.  (When disconnected graphs appeared, we discarded them without computing the Jacobians.)  The following table displays the proportions of these graphs with cyclic Jacobians.

\vspace{.25 cm}

\begin{center}
  \begin{tabular}{ |c | c | c | c | }
  \hline
    n$\setminus$ q & .3 & .5 & .7 \\ \hline
    15 & .784255 & .792895 & .775746 \\ \hline
    30 & .793807 & .793570 & .793375 \\ \hline
    45 & .793308 & .793962 & .793637 \\ \hline
    60 & .793436 & .793694 & .79354 \\ 
    \hline
  \end{tabular}
 \end{center}

\vspace{.25 cm}

\noindent We believe these data support Conjecture~\ref{C:cyc}, which predicts that the probability that $\Gamma(n,q)$ is cyclic tends to a limit slightly higher than $.7935$ as $n$ tends to infinity.

\subsection{Relative frequencies of $2$-groups with duality pairings}

We computed the Sylow $2$-subgroups with duality pairing for $10^5$ random graphs on $20$ vertices with edge probability $.5$, and compared the frequency of each group of size at most 8 with the frequency of the trivial group.  

In order to distinguish between the pairings that we observed, we recall the classification of finite abelian $p$-groups with duality pairing, following the presentation in \cite{Miranda84}, in which these results are attributed to Wall \cite{Wall64}.  Let $\mathcal{B}_p$ denote the semigroup of isomorphism classes of finite abelian $p$-groups with duality pairing under orthogonal direct sum.  The classification of symmetric bilinear forms on abelian $2$-groups is more complicated than for other $p$-groups \cite{Miranda84}.

\begin{prop}\label{mirandap2}
The semigroup $\mathcal{B}_2$ is generated by forms of the following types:
\begin{align*}
A_{2^r} & \text{ on } \Z/2^r\Z,\ r\ge 1; \langle 1,1\rangle = 2^{-r}, \\
B_{2^r} & \text{ on } \Z/2^r\Z,\ r\ge 2; \langle 1,1 \rangle = -2^{-r}, \\
C_{2^r} & \text{ on } \Z/2^r\Z,\ r\ge 3; \langle 1,1 \rangle = 5 \cdot 2^{-r},\\
D_{2^r} & \text{ on } \Z/2^r\Z,\ r\ge 3; \langle 1,1 \rangle = -5\cdot 2^{-r},\\
E_{2^r} & \text{ on } \Z/2^r\Z \times \Z/2^r \Z,\ r\ge 1; \langle e_i, e_j \rangle = \begin{cases} 0 & \text{if } i =j, \\ 2^{-r}& \text{if } i\neq j, \end{cases}\\
F_{2^r} &  \text{ on } \Z/2^r\Z \times \Z/2^r \Z,\ r\ge 2; \langle e_i, e_j \rangle = \begin{cases} 2^{-(r-1)} & \text{if } i =j, \\ 2^{-r}& \text{if } i\neq j \end{cases}.
\end{align*}
\end{prop}

\noindent The relations between these generators for $\mathcal{B}_2$ are complicated and we will not need them here.

\bigskip

\begin{center}
$\begin{array}{|c|c|c|c|c|}
\hline
\mathbf{(\Gamma, \delta)}&$\#$\textbf{Aut}\mathbf{(\Gamma, \delta)}&\textbf{Proportion}& \textbf{Observed}&\textbf{Expected}\\
& &\textbf{of Sample}&\textbf{Ratio}&\textbf{Ratio} \\

\hline
1&1&.419161&1&1\\
\hline
A_2&1&.210371&1.99249&2\\
\hline
A_4&2&.0518826&8.07903&8\\
\hline
B_4&2&.0522326&8.02489&8\\
\hline
A_2\oplus A_2 &2&.0517726&8.0962&8 \\
\hline
 E_4 &6&.0170709&24.5542&24 \\ 
\hline
A_8 &4&.0129706&32.3161&32\\
\hline
B_8&4&.0131007&31.9953&32 \\
\hline
C_8&4&.0132507&31.6332&32\\
\hline
D_8&4 &.0129206&32.4412&32\\
\hline
A_2 \oplus A_4 &2&.0259813&16.1332&16\\
\hline
A_2 \oplus A_2 \oplus A_2 &6&.00888044&47.2005&48 \\
\hline
\end{array}$
\end{center}

\bigskip

\noindent Here, the observed ratio is the observed frequency of the trivial group $.419161$ divided by the observed frequency of $(\Gamma, \delta)$, while the expected ratio is the factor $\# \Gamma \cdot \# \text{Aut}\;(\Gamma,\delta)$ predicted by Conjecture~\ref{C:basic}, in the special case where $F$ depends only on the Sylow $2$-subgroup of $\Gamma$ with its restricted pairing.

\subsection{Relative frequencies of $3$-groups with duality pairings}

We now recall the classification of finite abelian $p$-groups with duality pairing for odd primes $p$ \cite{Miranda84}.

\begin{prop}\label{miranda}
If $p$ is odd, the semigroup $\mathcal{B}_p$ is generated by cyclic groups with pairings of the following two types:
\begin{align*}
A_{p^r} & \text{ on } \Z/2^r\Z,\ r\ge 1; \langle 1,1\rangle = p^{-r}, and \\
B_{p^r} & \text{ on } \Z/2^r\Z,\ r\ge 1; \langle 1,1 \rangle = \alpha p^{-r},
\end{align*}
where $\alpha$ is a quadratic non-residue mod $p$.
\end{prop}

\noindent The semigroup $\mathcal{B}_p$ is not free on these generators; the relations are generated by $A_{p^r} \oplus A_{p^r} = B_{p^r} \oplus B_{p^r}$.

We computed the Sylow $3$-subgroup of the Jacobians on a sample of $10^5$ random graphs on 20 vertices with edge probability $.5$, and compared the relative frequency of each group with pairing of size at most 9 with the frequency of the trivial group.

\bigskip

\begin{center}$\begin{array}{|c|c|c|c|c|}
\hline
\mathbf{(\Gamma, \delta)}&$\#$\textbf{Aut}\mathbf{(\Gamma,\delta)}&\textbf{Proportion}& \textbf{Observed}&\textbf{Expected}\\
&&\textbf{of Sample}&\textbf{Ratio}&\textbf{Ratio} \\
\hline
1&1&.638566&1&1\\
\hline
A_3&2&.106104&6.01828&6\\
\hline
B_3&2&.106324&6.00583&6\\
\hline
A_9&2&.0349714&18.2597&18 \\
\hline
B_9&2&.0355414&17.9668&18\\
\hline
A_3 \oplus A_3 &8&.00905036&70.5569&72\\
\hline
A_3 \oplus B_3 &4&.0177307&36.0147&36 \\
\hline
\end{array}$
\end{center}

\bigskip

\noindent We found similar data for Sylow 5-subgroups and Sylow 7-subgroups of Jacobians of random graphs.

\subsection{Relative frequencies at two places}

Any duality pairing on a finite abelian group decomposes as an orthogonal direct sum of duality pairings on the Sylow $p$-subgroups.  Here we  give data supporting the hypothesis that the $p$-parts of Jacobians of random graphs are independent for distinct primes.

We computed the Sylow 2-subgroups and 3-subgroups of $2\cdot 10^5$ random graphs on 30 vertices with edge probability $.5$.  In 53366 cases, both of these subgroups were trivial.  The following table displays the ratio of the frequency of the trivial group to the frequency of each of the duality pairings on $\Z/2\Z \times \Z/2\Z \times \Z/3\Z$. 

\bigskip

\begin{center}
\begin{tabular}{|c|c|c|c|c|}
\hline
$\mathbf{(\Gamma, \delta)}$&\textbf{$\#\text{Aut}(\Gamma,\delta)$}& \textbf {Proportion} & \textbf{Observed} & \textbf{Expected} \\
& & \textbf{in Sample}& \textbf{Ratio} & \textbf{Ratio} \\
\hline
$A_2 \oplus A_2 \oplus A_3$&4& .005655 &47.1848&48\\
\hline
$A_2 \oplus A_2 \oplus B_3$&4& .005420 & 49.2306&48\\
\hline
$A_3 \oplus E_4$&12& .001865 &143.072&144\\
\hline
$B_3 \oplus E_4$&12& .001965 &135.791&144\\
\hline
\end{tabular}
\end{center}

\bigskip

\section{Acknowledgments}

The authors thank Matt Baker, Wei Ho,  Matt Kahle, and the referees.
  The fourth author was supported in part by NSF grant DMS-1068689 and NSF CAREER grant DMS-1149054.
The fifth author was supported by an American Institute of Mathematics Five-Year Fellowship and National Science Foundation grants DMS-1147782 and DMS-1301690.

\newcommand{\etalchar}[1]{$^{#1}$}
\providecommand{\bysame}{\leavevmode\hbox to3em{\hrulefill}\thinspace}
\providecommand{\MR}{\relax\ifhmode\unskip\space\fi MR }
\providecommand{\MRhref}[2]{%
  \href{https://urldefense.proofpoint.com/v2/url?u=http-3A__www.ams.org_mathscinet-2Dgetitem-3Fmr-3D-231&d=AwIFAg&c=-dg2m7zWuuDZ0MUcV7Sdqw&r=HXetpLHMphcM0tNPAuuorMjjwezLst1bZHvaSAGtD5M&m=-U9L2H6dbSz3KJGqUnsxhNvMKmlp5ifcI85Khhu_GXM&s=RLp4VGH2hR3q6HkuSf5_l1-F25RrJ7oTaaKdDOINeJs&e= }{#2}
}
\providecommand{\href}[2]{#2}


\begin{thebibliography}{99}

\bibitem[1]{BKLPR13}
M.~Bhargava, D.~Kane, H.~Lenstra, B.~Poonen, and E.~Rains, \emph{Modeling the
  distribution of ranks, {S}elmer groups, and {S}hafarevich-{T}ate groups of
  elliptic curves}, preprint arXiv:1304.3971, 2013.

\bibitem[2]{BannaiMunemasa98}
E.~Bannai and A.~Munemasa, \emph{Duality maps of finite abelian groups and
  their applications to spin models}, J. Algebraic Combin. \textbf{8} (1998),
  no.~3, 223--233.

\bibitem[3]{Cassels62}
J.~Cassels, \emph{Arithmetic on curves of genus {$1$}. {IV}. {P}roof of the
  {H}auptvermutung}, J. Reine Angew. Math. \textbf{211} (1962), 95--112.

\bibitem[4]{CohenLenstra84}
H.~Cohen and H.~Lenstra, \emph{Heuristics on class groups of number fields},
  Number theory, {N}oordwijkerhout 1983 ({N}oordwijkerhout, 1983), Lecture
  Notes in Math., vol. 1068, Springer, Berlin, 1984, pp.~33--62.

\bibitem[5]{CohenLenstra84b}
H.~Cohen and H.~Lenstra, Jr., \emph{Heuristics on class groups}, Number theory
  ({N}ew {Y}ork, 1982), Lecture Notes in Math., vol. 1052, Springer, Berlin,
  1984, pp.~26--36.

\bibitem[6]{arxiv}
J.~Clancey, N.~Kaplan, T.~Leake, S.~Payne, and M.~Wood, \emph{On a {C}ohen-{L}enstra heuristic for {J}acobians of random graphs}, arXiv:1402.5129.

\bibitem[7]{CLP13}
J.~Clancey, T.~Leake, and S.~Payne, \emph{A note on {J}acobians, {T}utte
  polynomials, and two-variable zeta functions of graphs}, Experiment. Math. \textbf{24} (2015), 1--7.
  
\bibitem[8]{ConwaySloane99}
J.~H. Conway and N.~J.~A. Sloane, \emph{Sphere packings, lattices and groups},
  third ed., Grundlehren der Mathematischen Wissenschaften [Fundamental
  Principles of Mathematical Sciences], vol. 290, Springer-Verlag, New York,
  1999, With additional contributions by E. Bannai, R. E. Borcherds, J. Leech,
  S. P. Norton, A. M. Odlyzko, R. A. Parker, L. Queen and B. B. Venkov.

\bibitem[9]{Delaunay01}
C.~Delaunay, \emph{Heuristics on {T}ate-{S}hafarevitch groups of elliptic
  curves defined over {$\Bbb Q$}}, Experiment. Math. \textbf{10} (2001), no.~2,
  191--196.

\bibitem[10]{Delaunay07}
\bysame, \emph{Heuristics on class groups and on {T}ate-{S}hafarevich groups:
  the magic of the {C}ohen-{L}enstra heuristics}, Ranks of elliptic curves and
  random matrix theory, London Math. Soc. Lecture Note Ser., vol. 341,
  Cambridge Univ. Press, Cambridge, 2007, pp.~323--340.
  

\bibitem[11]{Fulman14}
J.~Fulman, \emph{{H}all-{L}ittlewood polynomials and {C}ohen-{L}enstra
  heuristics for {J}acobians of random graphs}, to appear in Ann. Comb.  arXiv:1403.0473,
  2014.

\bibitem[12]{FriedmanWashington89}
E.~Friedman and L.~Washington, \emph{On the distribution of divisor class
  groups of curves over a finite field}, Th\'eorie des nombres ({Q}uebec, {PQ},
  1987), de Gruyter, Berlin, 1989, pp.~227--239.

\bibitem[13]{GJRWW}
L.~Gaudet, D.~Jensen, D.~Ranganathan, N.~Wawrykow, and T.~Weisman,
  \emph{Realization of groups with pairing as {J}acobians of finite graphs},
  preprint, arXiv:1410.5144, 2014.

\bibitem[14]{Lorenzini89}
D.~Lorenzini, \emph{Arithmetical graphs}, Math. Ann. \textbf{285} (1989),
  no.~3, 481--501.

\bibitem[15]{MacWilliams69}
J.~MacWilliams, \emph{Orthogonal matrices over finite fields}, Amer. Math.
  Monthly \textbf{76} (1969), 152--164.

\bibitem[16]{Miranda84}
R.~Miranda, \emph{Nondegenerate symmetric bilinear forms on finite abelian
  2-groups}, Trans. Amer. Math. Soc. \textbf{284} (1984), no.~2, 535--542.

\bibitem[17]{Shokrieh10}
F.~Shokrieh, \emph{The monodromy pairing and discrete logarithm on the
  {J}acobian of finite graphs}, J. Math. Cryptol. \textbf{4} (2010), no.~1,
  43--56.
  
  \bibitem[18]{Stanley12}
  R.~Stanley, \emph{Enumerative combinatorics. Volume 1}, second ed.,
  Cambridge Studies in Advanced Mathematics, 49.
  Cambridge University Press, Cambridge, 2012.

\bibitem[19]{Wagner00}
D.~Wagner, \emph{The critical group of a directed graph}, preprint,
  arXiv:math/0010241, 2000.

\bibitem[20]{Wall64}
C.~T.~C. Wall, \emph{Quadratic forms on finite groups, and related topics},
  Topology \textbf{2} (1964), 281--298.

\bibitem[21]{Wood14}
M.~Wood, \emph{The distribution of sandpile groups of random graphs}, preprint,
  arXiv:1402.5149, 2014.

\end{thebibliography}
\end{document}